\documentclass[a4paper]{article}
\usepackage{mathrsfs,amsfonts,amssymb,amsmath,amsthm,graphicx}
\usepackage[all]{xy}

\usepackage{fancybox}
\usepackage{color}

\newtheorem{theorem}{Theorem}[section]

\newtheorem{lemma}[theorem]{Lemma}
\newtheorem{proposition}[theorem]{Proposition}
\newtheorem{definition}{Definition}[section]
\newtheorem{remark}{Remark}[section]
\newtheorem{example}{Example}[section]

\DeclareMathOperator{\ENDO}{Endo} \DeclareMathOperator{\Ob}{Ob}
\DeclareMathOperator{\Mor}{Mor} \DeclareMathOperator{\Sz}{Sz}
\DeclareMathOperator{\id}{id} \DeclareMathOperator{\Inv}{Inv}
\DeclareMathOperator{\interior}{int}  
\DeclareMathOperator{\cl}{cl}

\ifx\pdfoutput\undefined
\usepackage[dvipdfm]{hyperref}
\else
\usepackage{hyperref}
\fi

\ifx\texonly\undefined\let\texonly\relax\fi
\ifx\endtexonly\undefined\let\endtexonly\relax\fi
\texonly
  \let\htmlonly\iffalse
  
\endtexonly

\begin{document}
\title{The Conley index for piecewise continuous maps}
\author{Kaihua Wang $^1$ and Xinchu Fu $^{2,}$\footnote{Corresponding author. E-mail address: xcfu@shu.edu.cn;
Tel: +86-21-66132664; Fax: +86-21-66133292}\\
 \small $^{1}$ School of Mathematics and Statistics, Hainan Normal University, Haikou 571158, P.R. China\\
 \small $^{2}$ Department of Mathematics, Shanghai University, Shanghai 200444, P.R. China}

\date{}

\maketitle

\begin{abstract}
\noindent This paper gives the definition of the Conley index for a
piecewise continuous map, which can be well defined on compatible
isolating neighborhoods with W\`{a}zewski property slightly weaker
than continuous situation.

\vspace{0.2cm}\noindent {\bf Key words:} Conley index; piecewise
continuous map; discontinuity; coding

\vspace{0.4cm}\noindent \textbf{MSC2010}: \emph{37B30; 18A99}

\end{abstract}

\section{Introduction}

The Conley index was first developed on flows \cite{CC_Morse}, then
it was extended to continuous maps \cite{MR_Cohomological, A.S_DSS,
FR_Shift}. The index has been extensively studied in recent years
\cite{[8],[5],[6],[17]}.

The Conley index is an abstract extension of the Morse index,
carrying both existence and stability information concerning a flow
or a discrete dynamical system. By illustrating these properties,
sufficient conditions for the presence of a bifurcation point for
$1$-dimensional flows were presented in \cite{FU_BIF,FU_BIF2}, and
more generally, by using the general Conley index theory (see
\cite{[l]}), sufficient conditions for the existence of a
bifurcation (sub-bifurcation) point of a family of flows on a
compact metric space were also given in \cite{FU_BIF,FU_BIF2}. There
have been many successful applications of the Conley index, e.g., to
the studies of travelling and shock waves \cite{[l]}, periodic
solutions for Hamiltonian systems \cite{[3]}, infinite-dimensional
semiflows \cite{[2]}, combustion problems \cite{[4]}, factoring
problems \cite{[18]}, theory and computation for set-valued maps
\cite{[7]}, rigorous numerics for dynamical systems
\cite{[9],[15],[19]}, characterization of chaos or spatial-temporal
chaos \cite{[10],[20]}, computer assisted rigorous proof for chaos
\cite{[l1],[16],[l2],[l3],[l4]}, and so on. Therefore in the study
of dynamical systems, including both flows and maps, the Conley
index has been a very useful tool.

Dynamics of discontinuous maps is an emerging research field of
dynamical systems
\cite{Chu&Lin88,hyperbolicpwl,Koc&al96,PA96,A.G_DPR,Buz,AdlKitTre99,FuAsh03}.
Discontinuities which may be caused by collision (impacting),
switching, overflow (round-off), quantization etc., are unavoidable
in both theory and applications. The main obstacle in studying
discontinuous systems is the lack of necessary mathematical tools.
In fact, due to discontinuities, most of analytical and topological
methods cannot be directly used. Furthermore, huge number of
discontinuous systems are even not measure-preserving, so ergodic
theory may invalidate, too. Thus it is very important to develop new
tools for discontinuous systems.

In this paper, we shall generalize the Conley index to discontinuous
case. Usually, real world systems do not contain much discontinuous
points, so  here we only study piecewise continuous maps defined on
finite partitions. We follow Goetz's construction \cite{A.G_DPR,
A.G_DPI} to lift a discontinuous system (a piecewise continuous map)
to a continuous one by graph of the coding map, then use the Conley
index for continuous case (for example \cite{MR_Cohomological,
A.S_DSS, FR_Shift} etc.) to define the index of the discontinuous
map. Due to discontinuity, one can imagine the index such defined
may be fairly weak. But it does generalize continuous index and
reserve some important information of the discontinuous system. The
properties of the index here is more like the situation in
\cite{D.M_CCom}, that is, the index is only defined on compatible
isolating neighbrohoods, and the W\`{a}zewski property holds, while
it is weaker than the usual one, such as the one given in
\cite{A.S_DSS}. Furthermore, due to the absence of continuity, which
leads to the lack of suitable homotopy, the continuation property is
lost.

\section{Preliminaries}

We use $\mathbb{R}^+, \mathbb{Z}$, $\mathbb{Z}^+$, $\mathbb{Z}^-$,
$\mathbb{N}$ to denote the sets of nonnegative real numbers,
integers, nonnegative integers, nonpositive integers and natural
numbers, respectively. For a topological space $X$ and a subset
$A\subset X$, we shall denote by $\cl_X(A)$, $\partial_X{A}$ and
$\interior_X(A)$ the closure, boundary and the interior of $A$ in
$X$, and will omit the subscript if no confusion may be caused.

For a topological space $X$, a finite partition
$\mathscr{A}=\{A_0,A_1,\ldots,A_{n-1}\}$ ($n\in \mathbb{Z}^+$) means
that (1) $A_i \subset X$  for any $i=0,1,\ldots,n-1$; (2) $A_i\cap
A_j=\emptyset$ whenever $i\neq j$; and (3)
$\bigcup^{n-1}_{i=0}A_i=X$. Each $A_i$ will be called a piece of the
partition.

Suppose $X$, $Y$ are two topological spaces. A map $f:X \rightarrow
Y$ is called a finite piecewise continuous map from $X$ to $Y$ if
there exists a finite partition $\mathscr{A}=\{A_0,\ldots,A_{n-1}\}$
of $X$ for some $n\in \mathbb{Z}^+$, such that $f_i:=f|\cl_X({A_i}):
\cl_X({A_i})\rightarrow Y$ is continuous for each $i=0,\ldots,n-1$.

Given a piecewise continuous map $f$, we may coarsen the partition
by unioning some of the pieces such that the new one is still
piecewise continuous with the same map. Thus we can coarsen the
partition finite times to get a new partition with minimum number of
pieces, i.e., it cannot coarsen anymore, and call it the minimal
partition with respect to $f$.

In the sequel, $X$ will be assumed to be a locally compact metric
space with metric $d_X$. We use PCM to denote piecewise continuous
maps (with minimal finite partition) from locally compact metric
space to itself, and call it piecewise continuous discrete dynamical
system or piecewise continuous system for short.

For a PCM $f:X\rightarrow X$ with partition
$\mathscr{A}=\{A_0,A_1,\ldots,A_{n-1}\}$, the discontinuity set
$D:=\bigcup_{i\neq j}\cl(A_i)\cap \cl(A_j)$ should be carefully
considered. We call a partition $\mathscr{B}=\{B_0,B
_1,\ldots,B_{m-1}\}$ of $X$ an adjoint partition of $\mathscr{A}$ if
(1) $m=n$; (2) There exits a permutation $\phi$ of
$\{0,1,\ldots,n-1\}$ such that $A_i-D=B_{\phi(i)}-D$ for every
$i=0,1,\ldots,n-1$; This is of course an equivalence relation, so we
can say $\mathscr{A}$, $\mathscr{B}$ are mutually adjoint. Now
define $g:X\rightarrow X$ by $g|B_{\phi(i)}:=f_i|B_{\phi(i)}$, then
we get a new PCM with partition $\mathscr{B}$, and call it an
adjoint map with $f$. Two adjoint maps may only differ on the
discontinuity set.

\section{Coding and continuous lifting}

To study a PCM, coding is a very helpful method.
\begin{definition}
Let $f:X\rightarrow X$ be a PCM with partition
$\mathscr{A}=\{A_0,\cdots,A_{n-1}\}$. We define the coding map
$\tau:X\rightarrow \Sigma^{+}_n:=\{0,1,\cdots,n-1\}^{\mathbb{Z}^+}$
by $\tau(x)=(k_0,k_1,\cdots)$, where $x\in X$, $f^i(x) \in A_{k_i}$,
$i\in \mathbb{Z}^+$ and $k_i \in \{0,\cdots,n-1\}$.
\end{definition}

Note that in the above definition, $\Sigma_n^+$ is exactly the
one-sided symbolic sequence space with state space
$\{0,1,\ldots,n-1\}$. We can give various product metric on it. For
example, one can define $d_{\Sigma}:\Sigma^+_n\times
\Sigma^+_n\rightarrow \mathbb{R}^+$ as
$d_{\Sigma}(x,y)=\max\{\frac{1}{n+1}|x_n\neq y_n\}$, where $x$, $y$
are two points in $\Sigma^+_n$.

\begin{remark}
The coding map $\tau$ is not continuous.
\end{remark}

Goetz indicated in his papers \cite{A.G_DPR,A.G_DPI} that any finite
piecewise isometry can be  lifted to a continuous map by using the
graph of coding map. This process can be extended to piecewise
continuous case.

Suppose $f:X\rightarrow X$ is a PCM and let $G_X:=\{(x,\tau(x))|x\in
X\}$ be the graph of $\tau:X\rightarrow \Sigma_n^{+}$ topologized by
the product metric $d_{G_X}: G_{X}\times G_{X}\rightarrow
\mathbb{R}^{+}$, where
$$d_{G_{X}}((x_1,\tau(x_1)),(x_2,\tau(x_2)))=\max\{d_X(x_1,x_2),d_\Sigma(\tau(x_1),\tau(x_2))\}.$$
Then we have:

\begin{proposition}
The graph map $\widetilde{f}:G_X\rightarrow G_X$, $(x,\tau(x))\mapsto (f(x),\tau(f(x)))$ is continuous
and the following diagram commutes:
\[
\xymatrix{
G_X \ar[r]^{\widetilde{f}} \ar[d]_{\pi} & G_X \ar[d]^\pi\\
X \ar[r]_f & X
}
\]
where $\pi$ is the natural projection on the first component.
\end{proposition}
\begin{proof}
The proof is much the same as that in \cite{A.G_DPR} , though $f$
here is more general.
\end{proof}

Let $\cl_{X\times \Sigma^+_n}(G_X)$ be the closure of $G_X$ in
$X\times \Sigma^+_n$, then we can naturally continuously extend
$\widetilde{f}:G_X\rightarrow G_X$ to
$\widehat{f}:\cl(G_X)\rightarrow \cl(G_X)$ by limit. That is, if
$(x,s)\in \cl(G_X)-G_X$ (then there must be a sequence $\{(x_n,\tau
(x_n))\}\subset G_X$ such that $(x_n,\tau (x_n))\rightarrow (x,s)$,
as $n\rightarrow \infty$), thus we define $\widehat{f}(x,s):=\lim
\limits_{n\rightarrow \infty} \widetilde f(x_n,\tau (x_n))$.
Unfortunately, for $\widehat{f}$ we do not have the commutative
diagram as $\widetilde{f}$ has, that is, in general $\pi \circ
\widehat{f}\neq  f \circ\pi$.

From now on, if $A\subset X$, we use $\widetilde{A}$ to denote the graph of $\tau$ restrict on $A$,
and $\widehat{A}$ to denote the closure of $\widetilde{A}$ in $\cl(G_X)$.

For any compact subset $N\subset X$, $\widetilde{N}$ may not
be a compact set in $G$. But for $\widehat{N}$ we have:

\begin{lemma}\label{lemma_comp}
For any compact subset $N$ of $X$, $\widehat{N}$ is compact in $\cl(G_X)$.
\end{lemma}
\begin{proof}
The result follows by noting that $N\times \Sigma^+_n$ is compact
and $\widehat{N}$ is closed in it.
\end{proof}

\section{Invariant sets and isolating neighborhoods}

Let $f:X\rightarrow X$ be a PCM and $A\subset X$, we define:\\
%\(\Inv_f^+(A):=\{x\in A:f^k(x)\in A, k\in \mathbb{Z}^+\};\)\\
\( \Inv_f^-(A):=\{x\in A:\exists\{x_k\}_{k\in \mathbb{Z}^-}\subset A
\mbox{ s.t. }
x_0=x \mbox{ and } f(x_k)=x_{k+1},\forall k \in \mathbb{Z}^{-}-\{0\}\};\)\\
\(\Inv_f(A):=\{x\in A: \exists \{x_k\}_{k \in \mathbb{Z}} \subset A
\mbox{ s.t. } x_0=x \mbox{ and }  f(x_k)=x_{k+1}, \forall k \in \mathbb{Z} \}.\)

\begin{definition}
Let $f:X\rightarrow X$ be a PCM, a subset $S$ of $X$
is called an invariant set (with respect to $f$) if $\Inv_f(S)=S$.
\end{definition}

\begin{definition}
A compact set $N \subset X$ is an isolating neighborhood with
respect to a PCM $f:X\rightarrow X$ if $\Inv_f(N)\subset \interior
N$.
\end{definition}

\begin{definition}
An invariant set $S$ is called an isolated invariant set with
respect to a PCM $f:X \rightarrow X$ if there exits an isolating
neighborhood $N$ such that $S=\Inv_f(N)$.
\end{definition}

\begin{remark}
When $f:X\rightarrow X$ is continuous and $N\subset X$ is compact,
$\Inv N$ is also compact. But for a PCM, this is not always true.
For example, consider $f:[0,1]\rightarrow [0,1]$, defined by
$f(x)=x$ when $x\in [0,\frac{1}{2})$ and
$f(x)=\frac{1}{2}x+\frac{1}{2}$ otherwise. Thus $[0,0.6]$ is a
compact set, while $\Inv_f([0,0.6])=[0,\frac{1}{2})$ is not compact
in $[0,1]$.
\end{remark}

\begin{lemma}\label{StoG}
Suppose $S\subset X$ is an invariant set with respect to PCM
$f:X\rightarrow X$, then $\widetilde{S}$ is an invariant set in $G_X$
with respect to $\widetilde{f}:G_X\rightarrow G_X$.
\end{lemma}

\noindent The proof is obvious.

\begin{lemma}\label{lemma_IN}
Suppose $S\subset X$ is an isolated invariant set with respect to a
PCM $f:X\rightarrow X$, then $\widehat{S}\subset \cl(G_X)$ is an
invariant set with respect to $\widehat{f}:\cl(G_X)\rightarrow
\cl(G_X)$.
\end{lemma}
\begin{proof}
We only need to show $\widehat{S}\subset
\Inv_{\widehat{f}}(\widehat{S})$, since the inverse inclusion is
obvious. For each $a\in \widehat{S}$, there are only two cases:

{\bf Case 1:} $a\in \widetilde{S}$, then by Lemma \ref{StoG}, $a\in
\Inv_{\widetilde{f}}(\widetilde{S})\subset
\Inv_{\widehat{f}}(\widehat{S})$.

{\bf Case 2:} $a=(x,u)\in \widehat{S}-\widetilde{S}$ for some $x\in
\cl(S)$. Then there must be a sequence $\{(x_n,\tau(x_n))\}_{n\in
\mathbb{N}}\subset \widetilde{S}$ s.t. $\lim_{n\rightarrow
\infty}(x_n,\tau(x_n))=(x,u)$. For each $n$, by Lemma \ref{StoG},
there exits $\{(x^{(i)}_n),\tau(x^{(i)}_n)\}_{i\in
\mathbb{Z}}\subset \widetilde{S}$, s.t. $\widetilde{f}
(x^{(i)}_n,\tau(x^{(i)}_n))=(x^{(i+1)}_n,\tau(x^{(i+1)}_n))$ and
$(x_n^{(0)},\tau(x^{(0)}_n))=(x_n,\tau(x_n))$. By the compactness of
$\widehat{S}$ ($S$ is isolated and Lemma \ref{lemma_comp}), we may
assume without loss of generality that
$(x^{(i)}_n,\tau(x^{(i)}_n))\rightarrow (x^{(i)},u^{(i)})$ for all
$i$ when $n\rightarrow \infty$. Therefore $\widehat{f}
(x^{(i)},u^{(i)})=\lim_{n\rightarrow
\infty}\widetilde{f}(x^{(i)}_n,\tau(x^{(i)}_n))=\lim_{n\rightarrow
\infty}(x^{(i+1)}_n,\tau(x^{(i+1)}_n))$ $=(x^{(i+1)},u^{(i+1)})$ and
$(x^{(0)},u^{(0)})$ $=\lim_{n\rightarrow
\infty}(x^{(0)}_n,\tau(x^{(0)}_n))=(x,u)$.
\end{proof}

Recall that for $\widehat{f}$ defined in previous section, due to
the lack of commutativity of the diagram, we should make our
definition about isolating neighborhoods slightly more strict.

\begin{definition}
A compact set $N \subset X$ is a compatible isolating neighborhood
with respect to PCM $f:X\rightarrow X$ if $N$ is an isolating neighborhood
and the intersection of $\partial N$ and discontinuous set contains no points of $\Inv^-_g(N)$, for any
$g$ adjoint with $f$.
\end{definition}

\begin{lemma}\label{lemma_CIN}
Suppose a compact subset $N\subset X$ is a compatible isolating neighborhood
of PCM $f:X\rightarrow X$, then $\widehat{N}$ is an isolating neighborhood of $\widehat{f}$.
\end{lemma}

\begin{proof}
By Lemma \ref{lemma_comp}, $\widehat{N}$ is compact in $\cl(G_X)$.
It can be trivially seen that $\partial \widetilde{N}$ contains no
points of $\Inv_{\widetilde{f}}(\widetilde{N})$. Suppose $(x,s)\in
\partial{(\widehat{N})}-\widetilde{N}$ and $(x,s)\in
\Inv_{\widehat{f}}(\widehat{N})$. Since
$\widehat{f}:\cl(G_X)\rightarrow \cl(G_X)$ is continuously extended
by $\widetilde{f}:G_X\rightarrow G_X$, we see that there is no point
$(x,\tau (x))\in G_X$ such that $\widehat{f}(x,\tau (x))=(x,s)$.
Thus the backward trajectory of $(x,s)$ are all contained in
$\widehat{N}-\widetilde{N}$, and one can easily check $x$ is in
$\Inv^-_g(N)$ for some $g$ adjoint with $f$.
\end{proof}

\section{The Conley index for continuous maps}

In this section, we briefly recall the definition of the Conley
index for a continuous map. There are various definitions, we choose
here Szymczak's construction  \cite{A.S_DSS}. Since continuous
dynamics can be viewed as PCM with just one piece,  concepts such as
isolating invariant sets and isolating neighborhoods defined in
previous section also valid for continuous situation, so we will not
restate these definitions.

To define the index, we require
to introduce some  categories.

Let $\mathscr{K}$ denote an arbitrary category. $\Ob(\mathscr{K})$
stands for its objects class, and $\Mor_{\mathscr{K}}(A,B)$ means
the morphism set from object $A$ to object $B$. We use $\mathscr{T}$
to denote topological space and continuous map category and
$\mathscr{T}'$ stands for its homotopy category.

Recall that the category $\ENDO(\mathscr{K})$ is defined by
\(\Ob(\ENDO(\mathscr{K})):=\{(X,f):X \in \Ob(\mathscr{K}), f \in \Mor_{\mathscr{K}}(X,X)\}\)
and
\( \Mor_{\ENDO(\mathscr{K})}((X,f),(X',f')):=\{\varphi \in \Mor_{\mathscr{K}}(X,X'):\varphi \circ f=f' \circ \varphi\}\).

Szymczak introduced the category of objects equipped with a
morphism over $\mathscr{K}$ in \cite{A.S_DSS}, we call it the Szymczak category
$\Sz(\mathscr{K})$, which is defined as:
the objects of $\Sz(\mathscr{K})$ are the same
as in $\ENDO(\mathscr{K})$
and
\[
\Mor_{\Sz(\mathscr{K})}((X,f),(X',f')):=\Big(\Mor_{\ENDO(\mathscr{K})}\big((X,f),(X',f')\big)\times
\mathbb{Z}^{+}\Big)/\sim
\]
where $(\varphi_1,n_1) \sim (\varphi_2,n_2)$ if and only if there
exists $k \in \mathbb{Z}^{+}$ such that the following diagram
commutes
\[
\xymatrix{ X \ar[rr] ^{f^{n_1+k}} \ar[d]^{f^{n_2+k}} && X
\ar[d]^{\varphi_2} \\
X \ar[rr] ^{\varphi_1} && X'}
\]

The relation ``$\sim$'' defined above is an equivalence relation. We
denote by $[\varphi,n]$ the morphism in
$\Mor_{\Sz(\mathscr{K})}((X,f),(X',f'))$. The identity morphisms are
given by $\id_{(X,f)}=[\id_X,0]$.

We define the composition of two morphisms $[\varphi,n]\in
\Mor_{Sz(\mathscr{K})}((X,f),(X',f'))$ and $[\varphi',n']\in
\Mor_{Sz(\mathscr{K})}((X',f'),(X'',f''))$ as $[\varphi',n']\circ
[\varphi,n]:=[\varphi'\circ \varphi,n'+n].$

Note that if $[g,n]\in \Mor_{\Sz(\mathscr{K})}((X,f),(X',f'))$, then
$[g,n]=[g\circ f^k,k+n]=[f'^{k}\circ g,k+n]$ for each $k \in
\mathbb{Z}^{+}$. In particular, if $m,n \in \mathbb{Z}^{+}$ and $h
\in Mor_{\mathscr{K}}(X,X)$, then $[h^m,m]=[h^n,n]=[\id_X,0]$.

Now suppose $h:X\rightarrow X$ is a continuous map and $S$ is an
isolating invariant set, then there exists an index pair. That is a
compact pair $P=(P_1,P_0)$ of $X$ which satisfies (1)
$S=\Inv_h(\cl(P_1-P_0))\subset \interior(P_1-P_0)$; (2) $x\in P_0$
implies $h(x)\not \in P_1-P_0$; (3) $x\in P_1$, $h(x)\not \in P_1$
implies $x \in P_0$.  For the index pair $P$, we can induce a map on
the quotient space $h_P: P_1/P_0\rightarrow P_1/P_0$ defined by
\[
h_p(\mbox{[}x\mbox{]}):=\left\{
\begin{array}{ll}
\mbox{[}h(x)\mbox{]}, & x\in h^{-1}(P_1)\\
\mbox{[}P_0\mbox{]}, & \mbox{otherwise}
\end{array}
\right.
\]
and call it the index map. One can verify that the index map is
continuous.

Szymczak's definition of the Conley index for an isolating
invariant set $S$ is the class of all objects isomorphic to $(P_1/P_0,h_P)$
in $\Sz(\mathscr{T}')$, and it is independent upon the choice of $P$.

\section{The Conley index for PCMs}\label{sec_index4pcms}

Suppose $f:X\rightarrow X$ is a PCM and $N$ is a compatible
isolating neighborhood with respect to $f$, then by Lemma
\ref{lemma_CIN}, $\widehat{N}$ is also an isolating neighborhood of
$\widehat{f}$. So $\Inv_{\widehat{f}}(\widehat{N})$ meets an index
pair $P=(P_1,P_0)$. Thus by definition, the index of
$\Inv_{\widehat{f}}(\widehat{N})$ is just the isomorphic class
$[P_1/P_0,(\widehat{f})_P]$ in $\Sz(\mathscr{T}')$. Since the Conley
index for a continuous map is independent upon the choice of index
pairs, we can give the following definition:
\begin{definition}
Let $f:X\rightarrow X$ be a PCM with minimal finite partition
$\mathscr{A} =\{A_0,A_1,\ldots,A_{n-1}\}$. If a compact set
$N\subset X$ is a compatible isolating neighborhood, then we define
the Conley index on $N$, denoted by $C(N)$ {\rm(}or $C(N,f)${\rm)},
to be the Conley index of $\Inv_{\widehat{f}}(\widehat{N})$.
\end{definition}

\begin{remark}
In continuous case, the Conley index is defined both on isolating
invariant sets and isolating neighborhoods. But our definition of
piecewise continuous case only define on compatible isolating
neighborhoods. That is, for different compatible isolating
neighborhoods with the same invariant set, their index may
different.
\end{remark}

\begin{proposition}[W\`{a}zewski property]
If $C(N)$ is nontrival, then at least one of the following is true:
\begin{enumerate}
\item $\Inv_f(N)\ne \emptyset$;
\item There exists some $g$ adjoint with $f$ ($g\ne f$),
such that $\Inv_g(N)\ne \emptyset.$
\end{enumerate}
\end{proposition}

\begin{proof}
$C(N)$ is nontrivial implies the Conley index (continuous case) of
$\widehat{N}$ is nontrivial. Thus, by \cite{A.S_DSS},
$\Inv_{\widehat{f}}(\widehat{N})\neq \emptyset$. According to Lemma
\ref{lemma_IN} we have $\widehat{\Inv_{f}(N)}\subset
\Inv_{\widehat{f}}(\widehat{N})\neq \emptyset$. Suppose
$\widehat{\Inv_f(N)}\neq \emptyset$, then we get $\Inv_f(N)\neq
\emptyset$ and (1) is true. Otherwise, we assume
$\widehat{\Inv_f(N)}=\emptyset$, then their must exists a point
$(x,s)\in \widehat{N}-\widetilde{N}$, such that $(x,s)\in
\Inv_{\widehat{f}}(\widehat{N})$. This means the forward and
backward trajectory of $(x,s)$ are both contained in $\widehat{N}$.
In the proof of lemma \ref{lemma_CIN}, we have already seen that
there is no point $(x,\tau(x))\in \widetilde{N}$ such that
$\widehat{f}(x,\tau(x))=(x,s)$. So the backward trajectory of
$(x,s)$ is contained in $\widehat{N}-\widetilde{N}$. But for $i\in
\mathbb{N}$, if $f^{i-1}(x,s)\in \widehat{N}-\widetilde{N}$, we have
either $f^i(x,s)\in \widehat{N}-\widetilde{N}$ or $f^i(x,s)\in
\widetilde{N}$. If for $i_0\in N$, $f^{i_0} (x,s)\in \widetilde{N}$,
then $f^{i}(x,s)$ can never be back to $\widehat{N}-\widetilde{N}$
for $i>i_0$. Thus either forward trajectory of $(x,s)$ is contained
in $\widehat{N}-\widetilde{N}$ or there exists $i_0\in \mathbb{N}$
such that $f^{i_0-1}(x,s)\in \widehat{N}-\widetilde{N}$ and
$f^{i_0}(x,s)\in \widetilde{N}$.  However both of these two cases
imply (2) is true.
\end{proof}

\begin{remark}
Other definitions of the Conley index for discrete semidynamical
systems can be chosen to define PCM's index. If one choose the
definition in  \cite{D.M_CCom}, we even don't need to take closure
of the graph of the coding map, but the index such defined can only
detect positive invariant sets.
\end{remark}

\begin{example}
Consider a piecewise continuous interval map $f:[-1,2]\rightarrow [-1,2]$
befined by
\[
f(x)=\left\{
\begin{array}{ll}
x+1/3, & x \in [-1,-1/3)\\
x+1, & x \in [-1/3,0)\\
x+2/3, & x \in [0,1/3)\\
x-1/3, & x \in [1/3,2/3)\\
-x+4/3, & x\in [2/3,1)\\
x-1, & x\in [1,4/3)\\
x-1/3, & x\in [4/3,2]
\end{array}.\right.
\]
It can be easily verify that $N=[-1/3,4/3]$ is an compatible
isolating neighborhood, and $P=(\widehat{N},\emptyset)$ is an index
pair in $\cl({G_X})$. So the index of $N$ is the isomorphic class
$[\mathcal{P}_6,g]$ in $\Sz(\mathscr{T}')$, in which $\mathcal{P}_6$
represents the pointed six points space, and $g$ is induced by
$\widehat{f}_P$ together with the homotopy equivalence from
$\widehat{N}/\emptyset$ to $\mathcal{P}_6$. It can be easily
verified that the index is nontrivial. Meanwhile, according to the
definition of $f$, we can see that $\Inv[-1/3,4/3]$ is nonempty.
Thus our definition of the Conley index for PCMs does keep some
existence information.
\end{example}

\section{Conclusions}

In this paper we lift a PCM to a continuous map, and then use the
definition of the Conley index for a continuous map to define the
index of a PCM, so we generalize the definition of the Conley index
to some discontinuous systems. Just like the similar situation
discussed in \cite{D.M_CCom}, where the Conley index for maps in
absence of compactness was introduced, here the index is only
defined on compatible isolating neighborhoods and the W\`{a}zewski
property is weaker than the usual index. However, it does keep some
important information, as being shown by an example given in
Section~\ref{sec_index4pcms}; moreover, for a continuous map the
minimal finite partition contains only one piece, this implies that
the lifting map is isomorphic to the original one in
$\ENDO(\mathscr{T})$, hence the index we defined exactly generalize
the continuous one. Unfortunately, due to the absence of continuity,
which leads to the lack of suitable homotopy of PCMs, the
continuation property of the Conley index is lost in piecewise
continuous situation.

\vspace{0.5cm}\noindent \textbf{Acknowledgment} This research was
supported jointly by NSFC Grant 10672146 and Shanghai Leading
Academic Discipline Project (S30104). The authors would like to
thank Congping Lin for discussions.

\end{document}